\documentclass[11pt]{amsart}
\usepackage{amssymb, amstext, amscd, amsmath, amssymb}
\usepackage{mathtools, xypic, paralist, color, dsfont, rotating}
\usepackage{verbatim}
\usepackage{enumitem}
\usepackage{enumerate}

\numberwithin{equation}{section}

\makeatletter
\def\@cite#1#2{{\m@th\upshape\bfseries%
[{#1\if@tempswa{\m@th\upshape\mdseries, #2}\fi}]}}
\makeatother
\theoremstyle{plain}
\newtheorem{theorem}{Theorem}[section]
\newtheorem{corollary}[theorem]{Corollary}
\newtheorem{proposition}[theorem]{Proposition}
\newtheorem{lemma}[theorem]{Lemma}
\theoremstyle{definition}
\newtheorem{definition}[theorem]{Definition}
\newtheorem{example}[theorem]{Example}

\newtheorem{remark}[theorem]{Remark}

\newtheorem*{acknow}{Acknowledgements}
\theoremstyle{remark}

\renewcommand{\qedsymbol}{{\vrule height5pt width5pt depth1pt}}
\mathtoolsset{centercolon}
  \newcommand{\A}{{\mathcal{A}}}
  \newcommand{\B}{{\mathcal{B}}}
  \newcommand{\C}{{\mathcal{C}}}
  \newcommand{\D}{{\mathcal{D}}}
  
  \newcommand{\F}{{\mathcal{F}}}

  \newcommand{\I}{{\mathcal{I}}}
  
  \newcommand{\K}{{\mathcal{K}}}
\renewcommand{\L}{{\mathcal{L}}}
  \newcommand{\M}{{\mathcal{M}}}

  \newcommand{\R}{{\mathcal{R}}}
\renewcommand{\S}{{\mathcal{S}}}
  \newcommand{\T}{{\mathcal{T}}}

\newcommand{\eps}{\varepsilon}
\def\al{\alpha}

\def\la{\lambda}
\def\La{\Lambda}

\def\si{\sigma}

\newcommand\wpi{\widetilde{\pi}}
\newcommand\wsi{\widetilde{\sigma}}

\newcommand{\bC}{\mathbb{C}}

\newcommand{\bN}{\mathbb{N}}

\newcommand{\foral}{\text{ for all }}
\newcommand{\qand}{\quad\text{and}\quad}

\newcommand{\qfor}{\quad\text{for}\ }

\newcommand{\ca}{\mathrm{C}^*}

\newcommand{\cenv}{\mathrm{C}^*_{\textup{env}}}
\newcommand{\tenv}{\mathcal{T}_{\textup{env}}}

\newcommand{\wt}{\widetilde}

\newcommand{\Alg}{\operatorname{Alg}}

\newcommand{\id}{{\operatorname{id}}}

\newcommand{\ran}{\operatorname{ran}}

\newcommand{\sca}[1]{\left\langle#1\right\rangle} 
\newcommand{\nor}[1]{\left\Vert #1\right\Vert} 
\newcommand{\ncl}[1]{\left[ #1 \right]^{-\|\cdot\|}} 
\newcommand{\wcl}[1]{\left[ #1 \right]^{-w^*}} 

\addtocontents{toc}{\protect\setcounter{tocdepth}{1}}

\begin{document}

\title[Strong Morita equivalence of operator spaces]{Strong Morita equivalence of \\ operator spaces}

\author[G.K. Eleftherakis]{George K. Eleftherakis}
\address{Department of Mathematics\\Faculty of Sciences\\University of Patras\\26504 Patras\\Greece}
\email{gelefth@math.upatras.gr}

\author[E.T.A. Kakariadis]{Evgenios T.A. Kakariadis}
\address{School of Mathematics and Statistics\\ Newcastle University\\ Newcastle upon Tyne\\ NE1 7RU\\ UK}
\email{evgenios.kakariadis@ncl.ac.uk}



\maketitle

\vspace{-.5cm}
\begin{center}
{\small \emph{Dedicated to the memory of Uffe Valentin Haagerup}}
\end{center}

\begin{abstract}
We introduce and examine the notions of strong $\Delta$-equiva\-len\-ce and strong TRO equivalence for operator spaces.
We show that they behave in an analogous way to how strong Morita equivalence does for the category of C*-algebras.
In particular, we prove that strong $\Delta$-equivalence coincides with stable isomorphism under the expected countability hypothesis, and that strongly TRO equivalent operator spaces admit a correspondence between particular representations.
Furthermore we show that strongly $\Delta$-equivalent operator spaces have stably isomorphic second duals and strongly $\Delta$-equivalent TRO envelopes.
In the case of unital operator spaces, strong $\Delta$-equivalence implies stable isomorphism of the C*-envelopes.
\end{abstract}

\section{Introduction}

In the 1950's Morita \cite{Mor58} introduced a notion of functorial equivalence for rings.
Morita's seminal work was popularized later by Bass \cite{Bas62}, and consists mainly of the Morita Theorems I, II, and III; see also \cite{Jac89}.
Motivated by the approach of Mackey \cite{Mac52, Mac53, Mac58} on representations of locally compact groups, Rieffel \cite{Rie74, Rie74-2} brought the analogues of Morita Theorems into the field of non-commutative geometry.
To this end Rieffel introduced a version of Morita equivalence for C*-algebras that implies stronger results.
Brown, Green and Rieffel \cite{Bro77, BGR77} introduced later what is sometimes called Morita Theorem IV: in the $\si$-unital case, strong Morita equivalence coincides with stable isomorphism.
The reader is also directed to the survey \cite{Rie82} by Rieffel.
The important aspect of strong Morita equivalence and stable isomorphism is the match of the intrinsic structure of C*-algebras that they induce.
From one point of view, strong Morita equivalence may be viewed as a generalised unitary equivalence (compare with equation (\ref{eq:tro}) that follows).

Their central role in representation theory has been a source of inspiration in the last 20 years for achieving Morita Theorems for a wider range of classes in operator theory.
The breakthrough in this direction came with the work of Blecher, Muhly and Paulsen \cite{BMP00} on operator algebras for Morita Theorems I and IV.
Later Blecher \cite{Ble01} added the relative Morita II and III parts.
Extensions to dual operator algebras were given by Blecher and  Kashyap \cite{BleKas08} and Kashyap \cite{Kas09}, for which the first three Morita Theorems were proven.
These works rely on the duality flavour of Morita equivalence, i.e. that the algebras $X$ and $Y$ can be decomposed into stabilized tensor products
\begin{equation}\label{eq:dec}
X \simeq M \otimes_Y N \qand Y \simeq N \otimes_X M
\end{equation}
of two appropriate bimodules $M$ and $N$.
The notion of tensor product (which varies each time) is used as a generalised multiplication rule.

Nevertheless, strong Morita equivalence in the case of C*-algebras requires that $M$ is a \emph{ternary ring of operators (TRO)}, i.e. $M M^* M \subseteq M$, and implies that $N = M^*$; see for example \cite[Section 2]{Ske09}.
Then a concrete interpretation hints that a second Morita theory for dual operator algebras is possible by defining $X$ to be \emph{equivalent} to $Y$ when there are completely isometric normal representations $\phi$ and $\psi$, and a TRO $M$ such that
\begin{equation}\label{eq:tro}
\phi(X) = \wcl{M \psi(Y) M^*} \qand \psi(Y) = \wcl{M^* \phi(X) M}.
\end{equation}
Considering this as the starting point, an alternative approach for dual operator algebras was developed.
The notion of $\Delta$-equivalen\-ce was introduced by the first author in \cite{Ele08}, and the first three Morita Theorems were proven for a certain category of modules over the algebras.
The appropriate Morita Theorem IV in this setting was later given by the first author and Paulsen \cite{ElePau08}.
A further generalization to the broader class of dual operator spaces was achieved by the first author with Paulsen and Todorov \cite{EPT10}.

Both extensions of Morita theory (versions (\ref{eq:dec}) and (\ref{eq:tro})) have advantages and a number of applications, e.g. \cite[Chapter 8]{BMP00}, \cite[Examples, p.p. 2405--2406]{BleKas08} and \cite[Section 3]{EPT10}.
As they both fit in the wider scheme of functorial equivalence they show resemblances and differences.
Relation (\ref{eq:tro}) implies relation (\ref{eq:dec}) as indicated by Blecher and Kashyap \cite[Acknowledgements]{BleKas08}.
The first author \cite{Ele10} has shown that relation (\ref{eq:tro}) is strictly stronger than relation (\ref{eq:dec}) for dual operator algebras.
This is because relation (\ref{eq:tro}) implies all four Morita Theorems, whereas relation (\ref{eq:dec}) does not imply in general a Morita IV Theorem.
In particular two nest algebras are equivalent in the (\ref{eq:dec}) version of \cite{BleKas08, Kas09} if and only if the nests are isomorphic, whereas they are equivalent in the (\ref{eq:dec}) version of \cite{Ele10} if and only if the isomorphism of the nests extends to an isomorphism of the von Neumann algebras they generate; see the work of the first author \cite{Ele13}.

The work on dual operator algebras was recently carried over to operator algebras by the first author \cite{Ele14}.
The appropriate $\Delta$-equivalence resembles to the (\ref{eq:tro}) relation where the closure is taken in the norm topology.
It is shown in \cite{Ele14} that this Morita context is strictly stronger than that of \cite{BMP00}, and in addition it satisfies the fourth element.

In the current paper we wish to move forward to the category of operator spaces.
By following the Morita context of (\ref{eq:tro}) we say that two operator spaces $X$ and $Y$ are \emph{strongly $\Delta$-equivalent} if there are completely isometric representations $\phi$ and $\psi$, and two TRO's $M_1$ and $M_2$ such that
\begin{equation}\label{eq:tro op sp}
\phi(X) = \ncl{M_2 \psi(Y) M_1^*} \qand \psi(Y) = \ncl{M_2^* \phi(X) M_1},
\end{equation}
(Definition \ref{D: Delta op sp}).
Our present aim is to research relations with the aforementioned equivalence relations.
We focus on the first and the fourth part of the suggested Morita theory and applications.
All results indicate that the (\ref{eq:tro op sp}) version has a canonical behaviour and blends well with previous Morita contexts.
The other parts of the Morita theory are to pursue elsewhere, as further focus is required for the analysis of the appropriate class of representations.

The first part is devoted in showing that strong $\Delta$-equivalence is indeed an equivalence relation (Theorem \ref{T: eq delta}).
To this end we also use a concrete version, that of strong TRO equivalence (Subsection \ref{Ss: TRO}).
It appears that the latter is more tractable and rather useful for our purposes.
One of the key tools is that strong TRO equivalence implies a bijection between non-degenerate representations of a certain class (Proposition \ref{P: bij rep}).
This result may be viewed as a part of Morita I Theorem.

In Section \ref{S: stable} we exploit the connection of strong $\Delta$-equivalence with stable isomorphism.
We show that in general stable isomorphism is stronger (Theorem \ref{T: stable gives delta}).
Nevertheless the two notions coincide when the operator spaces are separable (Corollary \ref{C: sep}) or unital (Corollary \ref{C: unital}).
These follow as consequences of two key results (Theorem \ref{T: key} and Corollary \ref{C: si-unital}) concerning a countability hypothesis on the related C*-algebras
\[
A = \ncl{M_2M_2^*}, \, B = \ncl{M_1M_1^*}, \, C = \ncl{M_2^*M_2}, \, D = \ncl{M_1^*M_1}.
\]

In Section \ref{S: app} we give applications of our results and connections with the literature.
First we prove that if two operator algebras with contractive approximate identities are strongly $\Delta$-equivalent as operator spaces then they are strongly Morita equivalent in the sense of Blecher, Muhly and Paulsen \cite{BMP00} (Theorem \ref{T: delta gives BMP}).
In fact this implication is strict (Remark \ref{R: delta gives BMP strict}).
As a corollary we then get that two C*-algebras are strongly $\Delta$-equivalent as operator spaces if and only if they are strongly Morita equivalent in the sense of Rieffel \cite{Rie74} (Corollary \ref{C: same in C*}).
Furthermore we show that the second duals of strongly $\Delta$-equivalent operator spaces are stably isomorphic in the sense of \cite{EPT10} (Theorem \ref{T: second dual}).
Finally we examine the impact of strong $\Delta$-equivalence to their TRO envelopes (or their C*-envelopes in the unital case) in the sense of Hamana \cite{Ham99} and Arveson \cite{Arv69}.
In particular we show that strong $\Delta$-equivalence of the operator spaces $X$ and $Y$ implies strong $\Delta$-equivalence of their TRO envelopes $\tenv(X)$ and $\tenv(Y)$ (Theorem \ref{T: tro env}).
Similarly, stably isomorphic unital operator spaces admit stably isomorphic C*-envelopes (Theorem \ref{T: cenv}).

\section{Preliminaries}\label{S: Pre}

All operators act on Hilbert spaces and are bounded.
Therefore when we write $\B(K_1, K_2)$ we will automatically assume that the spaces $K_1$ and $K_2$ are Hilbert spaces.
All limits are taken in the norm topology (either that of $\B(K_1, K_2)$ or that of $K_1$ or $K_2$), unless otherwise specified.
If a space $X$ acts on a Hilbert space $K$ then we will write $[XK]$ for the closed subspace of $K$ generated by the linear span of $x \xi$ for all $x \in X$ and $\xi \in K$.
If a space $A$ acts on a normed space $X$ then we write $\ncl{AX}$ for the closed subspace of $X$ generated by the linear span of $a x$ for all $a \in A$ and $x \in X$.
The reader is addressed to \cite{BleLeM04, EffRua00, Pau02, Pis03} for the pertinent details on operator spaces.

An \emph{operator space} $X$ is a norm closed subspace of $\B(K_1, K_2)$.
As such it inherits a matricial norm structure from $\B(K_2 \oplus K_1)$.
If $K_1 = K_2$ and $I_{K_1} \in X$ then the operator space $X$ is called \emph{unital}.
For this paper the morphisms of the operator spaces are the completely contractive linear maps; isomorphisms are then the completely isometric surjective maps.

Similarly an \emph{operator algebra} is a subalgebra of some $\B(K)$ and an \emph{operator system} is a unital selfadjoint subspace of some $\B(K)$.
The morphisms are respectively the completely contractive algebraic homomorphisms and the unital completely positive maps.
The term \emph{c.a.i.} stands for a contractive approximate identity of an operator algebra.

A completely contractive map $\phi \colon X \to \B(K_1, K_2)$ will be called \emph{non-degenerate} if $K_2 = [\phi(X) K_1]$ and $K_1 = [\phi(X)^* K_2]$.
If $\phi \colon X \to B(K_1, K_2)$ is degenerate then we may pass to the non-degenerate completely contractive map $\phi' := P_{K_2'} \phi(\cdot) |_{K_1'}$ for $K_2' = [\phi(X) K_1]$ and $K_1' = [\phi(X)^* K_2]$; then
\[
\phi(x) = \begin{bmatrix} \phi'(x) & 0 \\ 0 & 0 \end{bmatrix} \foral x \in X.
\]
Hence $\phi$ is completely isometric if and only if $\phi'$ is completely isometric.

\subsection{Operator bimodules}

We will require some notation about the class of bimodules.
Let $A$ and $B$ be operator algebras with c.a.i.'s.
We say that an operator space $X$ is an \emph{operator $A$-$B$-bimodule} if there exist completely contractive bilinear maps $A\times X\rightarrow X$ and $X\times B\rightarrow X$.
In this case we write ${}_A X_B$.
A bimodule ${}_A X_B$ will be called \emph{non-degenerate} if both $A$ and $B$ act non-degenerately on $X$, i.e. $\ncl{AX} = X$ and $\ncl{XB} = X$.
If in particular $A$ and $B$ are C*-algebras then we will say that ${}_A X_B$ is a \emph{C*-bimodule}.
We will be mainly interested in non-degenerate C*-bimodules.

The morphisms $(\pi, \phi, \si) \colon {}_A X_B \to \B(K_1, K_2)$ of C*-bimodules consist of bimodule maps such that $\phi \colon X \to \B(K_1, K_2)$ is a completely contractive map, and $\pi \colon A \to \B(K_2)$ and $\si \colon B \to \B(K_1)$ are $*$-representations.
A bimodule map $(\pi, \phi, \si)$ will be called \emph{non-degenerate} if $\pi$, $\phi$ and $\si$ are non-degenerate.
Two bimodule maps $(\pi,\phi,\si) \colon {}_A X_B \to \B(K_1, K_2)$ and $(\pi',\phi',\si') \colon {}_A X_B \to \B(K_1', K_2')$ are called \emph{unitarily equivalent} if there are unitaries $V \in \B(K_1, K_1')$ and $U \in \B(K_2, K_2')$ such that
\[
(\pi',\phi',\si') = (U \pi(\cdot) U^*, U \phi(\cdot) V^*, V \si(\cdot) V^*).
\]
For every non-degenerate C*-bimodule ${}_A X_B$ there exist a complete isometric map $\phi \colon X \rightarrow B(K_1, K_2)$, and faithful representations $\pi \colon A \rightarrow B(K_2)$ and $\sigma \colon B \rightarrow B(K_1)$ such that
\[
\phi(axb)= \pi (a)\phi (x)\sigma (b) \foral  a\in A, x\in X, b \in B;
\]
see for example \cite[Theorem 3.3.1]{BleLeM04}.
Every such triple $(\pi, \phi ,\sigma )$ is called \emph{a faithful CES representation of ${}_A X_B$}.
We will use the following observations concerning non-degenerate bimodule maps.

\begin{lemma}\label{L: nd iff}
Let ${}_A X_B$ be a non-degenerate C*-bimodule.
Then a completely contractive bimodule map $(\pi, \phi, \si) \colon {}_A X_B \to \B(K_1, K_2)$ is non-degenerate if and only if $\phi \colon X \to \B(K_1, K_2)$ is non-degenerate.
\end{lemma}

\begin{proof}
The ``only-if'' part is trivial.
For the ``if'' part we have that
\[
[\pi(A) K_2] = [ \pi(A) \phi(X) K_1 ] = [\phi(AX) K_1] = [\phi(X) K_1] = K_2,
\]
since $\ncl{AX} = X$.
A symmetrical computation applies for $\si$.
\end{proof}

We will use the following construction in the case of non-degenerate C*-bimodules to pass to non-degenerate completely contractive bimodule maps.
Let ${}_A X_B$ be a C*-bimodule and fix a completely contractive bimodule map $(\pi,\phi,\si) \colon {}_A X_B \to \B(K_1, K_2)$.
Let $K_2' = [\phi(X) K_1]$ and $K_1' = [\phi(X)^* K_2]$ and define the compression $\phi'(\cdot) = P_{K_2'} \phi(\cdot) |_{K_1'}$.
A direct computation shows that $K_1'$ is reducing for $\pi$ and that $K_2'$ is reducing for $\si$.
Thus we may define the non-degenerate sub-representations $\pi' = \pi|_{K_2'}$ and $\si'= \si|_{K_1'}$ of $A$ and $B$, respectively.
Then Lemma \ref{L: nd iff} implies that the completely contractive bimodule map $(\pi', \phi', \si') \colon {}_A X_B \to \B(K_1', K_2')$ is non-degenerate.

\begin{proposition}\label{P: nd compression}
Let ${}_A X_B$ be a non-degenerate C*-bimodule.
Fix a completely contractive bimodule map $(\pi, \phi, \si) \colon {}_A X_B \to \B(K_1, K_2)$ and let the non-degenerate compression $(\pi',\phi', \si') \colon {}_A X_B \to \B(K_1',K_2')$ as constructed above.
Then $\phi$ is a complete isometric map if and only if $\phi'$ is a complete isometric map.
If, in addition, $A$ and $B$ act faithfully on $X$ then $(\pi, \phi, \si)$ is a faithful CES representation if and only if $(\pi', \phi', \si')$ is a faithful CES representation.
\end{proposition}

\begin{proof}
We have already mentioned that the non-degenerate compression $\phi'$ is completely isometric if and only if so is $\phi$.
For the second part suppose that $(\pi, \phi, \si)$ is a faithful CES representation.
If $\pi'(a) = 0$ then $\phi'(a x) = \pi'(a) \phi'(x) = 0$ for all $x \in X$.
Thus we get that $a = 0$ since $\phi'$ is injective and $A$ acts faithfully on $X$.
Similarly we obtain that $\si'$ is faithful.
The converse implication is trivial.
\end{proof}

\begin{remark}\label{R: faithfulness}
Faithfulness of the action in the second part of Proposition \ref{P: nd compression} is necessary.
Indeed, there are non-degenerate bimodule maps $(\pi, \phi, \si) \colon {}_A X_B \to \B(K_1, K_2)$ with $\phi$ being a complete isometry but $\pi$ and $\si$ are not faithful.
For such an example let $X = \bC$ and $A = B = \bC^2$ such that
\[
(a_1, a_2) \cdot \xi \cdot (b_1, b_2) = a_1 \xi b_1.
\]
Then we have that $\ker \pi  = \ker \si = \{(0,c) \mid c \in \bC\}$ for the representations $\pi(a_1, a_2) = a_1$, $\si(b_1, b_2) = b_1$ and $\phi(\xi) = \xi$.
\end{remark}

\subsection{Ternary rings of operators}

An operator space $M$ is called a \emph{ternary ring of operators (TRO)} if $MM^*M \subseteq M$.
It then follows that $M$ is an $A$-$B$-equivalence bimodule in the sense of Rieffel \cite{Rie74} for the C*-algebras
\[
A = \ncl{MM^*} \qand B = \ncl{M^*M}.
\]
Indeed, the C*-algebras $A$ and $B$ act non-degenerately and faithfully on $M$.
Consequently the C*-algebras $A$ and $B$ attain c.a.i.'s which endow $M$ with a left and a right c.a.i., respectively.
Thus a TRO $M$ satisfies $\ncl{MM^*M} = M$.
In particular, there are two nets $(a_t)$ and $(b_\la)$ where
\begin{align*}
a_t =
\sum_{i=1}^{l_t} m_i^t (m_i^t)^*,
\qand
b_\lambda =
\sum_{i=1}^{k_\lambda } (n_i^\lambda)^* n_i^\lambda
\end{align*}
for some $m_i^t, n_j^\lambda \in M$ such that $[m_1^t, m_2^t, \dots, m_{l_t}^t]$ and $[(n_1^\la)^*, (n_2^\la)^*, \dots, (n_{k_\la}^\la)^*]$ are contractions, and
\[
\lim _t a_t m = m \qand \lim_\lambda m b_\lambda = m
\]
for all $m\in M$; see for example the proof of \cite[Theorem 6.1]{BMP00}.
In the case where $A$ and $B$ admit countable approximate identities then the corresponding nets can be replaced by sequences
\[
a_k = \sum_{i=1}^k m_i m_i^*
\qand
b_k = \sum_{i=1}^k n_i^* n_i.
\]
This follows a standard trick for TRO's; see for example \cite[Lemma 2.3]{Bro77}.

It is a standard fact that a $*$-representation $\si \colon B \to \B(K)$ induces a completely contractive map $t \colon M \to \B(K, M \otimes_\si K)$ by left creation operators, i.e. $t(m) \xi = m \otimes \xi$, where $M \otimes_\si K$ is the Hilbert space with
\[
\sca{m \otimes \xi, n \otimes \eta} = \sca{\xi, \si(m^*n) \eta} \foral m,n \in M, \xi, \eta \in K.
\]
Consequently the mapping $\pi(a) m \otimes \xi = (am) \otimes \xi$ defines a $*$-representation $\pi \colon A \to  \B(M \otimes_\si K)$.
In particular $t$ is a TRO morphism of $M$ with $t(m) t(n)^* = \pi(m n^*)$ and $t(m)^* t(n) = \si(m^*n)$ for all $m,n \in M$.

We fix notation for TRO envelopes.
We will use this terminology instead of that of the injective triple extension.
For full details see \cite[Section 8.3]{BleLeM04}.
For a complete isometric map $i \colon X \to Y$ into a TRO $Y$, let $\T(i(X))$ be the TRO spanned by
\[
i(x_1) i(x_2)^* i(x_3) i(x_4)^* \cdots i(x_{2n})^* i(x_{2n+1}) \text{ for } n\geq 0 \text{ and } x_1,\dots, x_{2n+1} \in X,
\]
and their limits.
We say that $(\T(i(X)),i)$ is \emph{a TRO extension of $X$}.
There is a particular embedding of $X$ in the injective envelope $\I(\S(X))$ of an operator system $\S(X)$.
If $X$ is unital then we set $\S(X) = X + X^*$, and when $X$ is not unital we set
\[
\S(X) = \{ \begin{bmatrix} \la & x_1 \\ x_2 & \mu \end{bmatrix} \mid x_1, x_2 \in X, \la, \mu \in \bC\},
\]
i.e. \emph{the Paulsen system}.
The injective envelope $\I(X)$ of $X$ is then the corner $\I_{12}(\S(X))$, and in particular $\I(X + X^*)$ when $X$ is unital.
The TRO extension of $X$ generated in $\I(\S(X))$ will be denoted by $\tenv(X)$.
In particular $\tenv(X)$ carries the following universal property, due to Hamana \cite{Ham99}:
given any TRO extension $(Z,j)$ of $X$ there exists a necessarily unique and surjective triple morphism $\theta \colon Z \to \tenv(X)$ such that $\theta(j(x)) = x$.
The TRO space $\tenv(X)$ is called \emph{the TRO envelope of $X$}.
If in addition $X$ is a C*-bimodule over $A$ and $B$ then so is every TRO extension $(Z,j)$ of $X$, and the TRO morphism $\theta \colon Z \to \tenv(X)$ is a bimodule map over $A$ and $B$.

If $X$ is unital then the embedding $X \hookrightarrow \I(X)$ is unital and the Choi-Effros construction endows $\I(X)$ with a C*-algebraic structure.
Then the TRO envelope is a C*-algebra and is denoted by $\cenv(X)$.
The existence of the C*-envelope is again due to Hamana \cite{Ham79} and follows the program established by Arveson \cite{Arv69} (see also \cite{KakPet14} for another interpretation).
For an alternative proof of Hamana's Theorem \cite{Ham79} by using boundary subsystems the reader is directed to the work of the second author \cite{Kak11-2}.
In fact the arguments of \cite{Kak11-2} suffice to show the existence of the TRO envelope as well; even though this is not mentioned therein.

\subsection{Multipliers of operator spaces}

We require notation about multipliers on operator spaces.
For full details see \cite[Section 4.5]{BleLeM04}.
Fix an operator space $X$.
Recall the decomposition
\[
X \hookrightarrow \S(X) \hookrightarrow \I(\S(X)) = \begin{bmatrix} \I_{11}(X) & \I(X) \\ \I(X)^* & \I_{22}(X) \end{bmatrix}.
\]
We write $\M_l(X)$ for \emph{the operator space of left multipliers of $X$}.
Then $\M_l(X)$ is completely isometrically isomorphic as an operator algebra to
\[
IM_l(X) := \{s \in \I_{11}(X) \mid sX \subseteq X\}.
\]
In particular if $X \subseteq \B(K)$ is an operator algebra with c.a.i. then $\M_l(X)$ is completely isometrically isomorphic as an operator algebra to $\{s \in \B(K) \mid s X \subseteq X\}$.
The diagonal of $\M_l(X)$ is denoted by $\A_l(X)$ and is a unital C*-algebra.
In particular $\A_l(X)$ is $*$-isomorphic to the C*-algebra
\[
\{s \in \I_{11}(X) \mid sX \subseteq X \text{ and } s^*X \subseteq X \}.
\]
Similar facts hold for the operator space $\M_r(X)$ of the right multipliers and its diagonal $\A_r(X)$.
The maps
\[
\A_l(X) \times X \to X : (l,x) \mapsto lx, \qand X \times \A_r(X) \to X : (x,r) \mapsto xr
\]
are completely contractive and $X$ is an $\A_l(X)$-$\A_r(X)$-bimodule.
In particular the C*-bimodule ${}_{\A_l(X)} X_{\A_r(X)}$ is non-degenerate and $\A_l(X)$ and $\A_r(X)$ act faithfully on $X$.

Let $X$ and $Y$ be operator spaces.
An \emph{oplication of $Y$ on $X$} is a completely contractive bilinear map $m \colon Y \times X \to X$ such that there is a net $(e_i)$ of contractions in $Y$ for which $\lim_i m(e_i, x) = x$ for all $x \in X$.
The term is a short for ``operator multiplication''.
It is not immediate, but given an oplication of $Y$ on $X$ there exists a (necessarily unique) completely contractive map $\theta \colon Y \to \M_l(X)$ such that $\theta(y)(x) = m(y, x)$ for all $y\in Y, x \in X$.
If $e_i = e$ for all $i$ then this map sends $e$ to $I_X$.
In particular if $Y$ is an algebra (resp. C*-algebra) then $\theta$ is a homomorphism (resp. $*$-homomorphism into $\A_l(X)$).
Similar arguments follow for the right version of oplications.
The reader is addressed to \cite[Section 4.6.1]{BleLeM04} for the required details.

\section{Strong $\Delta$-Equivalence} \label{S: Delta}

\subsection{Strong TRO equivalence for operator spaces}\label{Ss: TRO}

Let us give the concrete version of the Morita context we are going to study.

\begin{definition}\label{D: TRO op sp}
Let $X\subseteq B(K_1, K_2)$ and $Y\subseteq B(H_1, H_2)$ be operator spaces.
We say that $X$ and $Y$ are \emph{strongly TRO equivalent} if there exist TRO's $M_1\subseteq B(H_1, K_1)$ and $M_2 \subseteq \B(H_2, K_2)$ such that
\[
X= \ncl{M_2YM_1^*} \qand Y=\ncl{M_2^*XM_1}.
\]
\end{definition}

It is immediate that strong TRO equivalence involves non-degenerate C*-bimodules.
Indeed if $X$ and $Y$ are strongly TRO equivalent by $M_1$ and $M_2$ then $X$ is an $A$-$B$-bimodule for the C*-algebras
\[
A = \ncl{M_2M_2^*}, \, B = \ncl{M_1M_1^*},
\]
and $Y$ is a $C$-$D$-bimodule for the C*-algebras
\[
C = \ncl{M_2^*M_2}, \, D = \ncl{M_1^*M_1}.
\]
It is immediate that
\[
\ncl{A X} = X = \ncl{X B} \qand \ncl{C Y} = Y = \ncl{Y D}.
\]
Therefore ${}_A X_B$ and ${}_C Y_D$ are non-degenerate as C*-bimodules.
Moreover if $X$ acts non-degenerately then $A$ and $B$ act faithfully on $X$.
Indeed if $ax = 0$ for all $x \in X$ then $a \xi = 0$ for all $\xi \in K_2 = [X K_1]$.
A similar argument holds for $B$ by using adjoints.
Let us give some examples of strongly TRO equivalent operator spaces.

\begin{example}\label{E: sMe}
Strong TRO equivalence generalises the usual strong Morita equivalence in the case of C*-algebras.
Indeed if $M$ is an equivalence bimodule over the C*-algebras $X$ and $Y$, then $X = \ncl{MM^*}$ and $Y = \ncl{M^*M}$ are strongly TRO equivalent by $M_1 = M_2 = M^*$.
\end{example}

\begin{example}\label{E: trivial}
We can easily construct strongly TRO equivalent operator spaces.
Take an arbitrary operator space $X_0\subseteq B(K_1, K_2)$ and TRO's $M_i\subseteq B(H_i, K_i)$ for $i=1,2$.
Then the operator spaces
\[
X= \ncl{M_2M_2^*X_0M_1M_1^*} \qand Y = \ncl{M_2^*X_0M_1}
\]
are automatically strongly TRO equivalent.
\end{example}

\begin{example}\label{E: compacts}
We write $\K(X)$ for the operator subspace of the compact operators that the operator space $X\subseteq B(K_1, K_2)$ contains.
Moreover $\F(X)$ (resp. $\R(X)$) denotes the operator subspace of the finite rank operators (resp. rank $1$ operators) that $X$ contains.

Let $X, Y$ be weakly TRO equivalent dual operator spaces in the sense of \cite{EPT10}, which act non-degenerately on Hilbert spaces.
That is, there exist TRO's $M_1$ and $M_2$  such that $X = \wcl{M_2 Y M_1^*}$ and $Y = \wcl{M_2^* X M_1}$.
Then $\K(X)$ and $\K(Y)$ are strongly TRO equivalent.
Similarly $\F(X)$ (resp. $\R(X)$) and $\F(Y)$ (resp. $\R(Y)$) are strongly TRO equivalent.

We will only show the case of $\K(X)$ and $\K(Y)$.
Analogous arguments settle the other cases.
Since $M_2^* X M_1 \subseteq Y$ then we get $M_2^* \K(X) M_1\subseteq \K(Y)$.
Similarly  we have that $M_2 \K(Y) M_1^*\subseteq \K(X)$ and thus we obtain
\[
M_2^*M_2 \K(Y) M_1^*M_1 \, \subseteq \, M_2^* \K(X) M_1.
\]
Let $(e_i)$ be a c.a.i. of $\ncl{M_2^*M_2}$.
Since $M_2$ acts non-degenerately, then $(e_i)$ converges in the strong operator
topology to the identity operator of the algebra $\wcl{M_2^* M_2}$.
We conclude that $k = \lim_i e_i k$ for every $k \in \K(Y)$.
Similarly we get that $k = \lim_i  k f_i $, for all $k \in \K(Y)$, for a c.a.i. $(f_i)$ of $\ncl{M_1^*M_1}$.
Therefore we get that
\[
\K(Y) \subseteq \ncl{M_2^*M_2 \K(Y) M_1^*M_1} \subseteq \ncl{M_2^* \K(X) M_1},
\]
and thus $\K(Y) = \ncl{M_2^* \K(X) M_1}$; likewise $\K(X) = \ncl{M_2 \K(Y) M_1^*}$.
\end{example}

\begin{example}\label{E: lattices}
Let $\L_1$ and $\L_2$ be reflexive lattices, and let $A=\Alg(\L_1)$ and $B=\Alg(\L_2) $ be the corresponding algebras; see \cite{Dav88} for the pertinent definitions.
If $\theta \colon \L_1^{\prime \prime} \rightarrow \L_2^{\prime \prime}$ is a $*$-isomorphism such that $\theta (\L_1)=\L_2$, then
the algebras $A$ and $B$ are weakly TRO equivalent in the sense of \cite{Ele12}.
Therefore by Example \ref{E: compacts} we get that $\K(A)$ and $\K(B)$ are strongly TRO equivalent.
The same holds for $\F(A)$ and $\F(B)$, as well as for $\R(A)$ and $\R(B)$.
\end{example}

Let us continue with the main results of this subsection.
The next theorem follows a similar reasoning as in \cite[Theorem 2.1]{Ele14}.

\begin{theorem}\label{T: eq rel}
Strong TRO equivalence of operator spaces is an equivalence relation.
\end{theorem}

\begin{proof}
For reflexivity, observe that $X \subseteq \B(K_1, K_2)$ is strongly TRO equivalent to itself by $M_2=\bC I_{K_2}$ and $M_1=\bC I_{K_1}$.
The relation is symmetric as the adjoints of a TRO form a TRO.
For transitivity, fix the operator spaces $X \subseteq B(K_1, K_2)$, $Y \subseteq B(H_1, H_2)$ and $Z \subseteq B(W_1, W_2)$ such that
\[
X = \ncl{M_2 Y M_1^*}, \, Y = \ncl{M_2^* X M_1} = \ncl{N_2 Z N_1^*}, \, Z = \ncl{N_2^* Y N_1}
\]
for some TRO's $M_i$ and $N_i$ for $i = 1,2$.
Then it can be shown that $X = \ncl{L_2 Z L_1^*}$ and $Z = \ncl{L_2^* X L_1}$ for the TRO's
\[
L_1 := \ncl{M_1 D_1 N_1} \qand L_2 := \ncl{M_2 D_2 N_2}
\] 
where $D_i : = \ca(\{M_i^*M_i \cup N_i N_i^*\})$ for $i=1,2$.
The proof follows a similar reasoning to \cite[Theorem 2.1]{Ele14}, by using that $\ncl{D_2 Y} = \ncl{Y D_1} = Y$, and it is left to the reader.
\end{proof}

For the next result recall the notion of unitary equivalence between completely contractive maps of C*-bimodules.
The proof uses the equivalence of representations between strong Morita equivalent C*-algebras \cite{Rie74}.

\begin{proposition}\label{P: bij rep}
Let $X$ and $Y$ be strongly TRO equivalent by $M_1$ and $M_2$, and let
\[
A = \ncl{M_2M_2^*}, \, B = \ncl{M_1M_1^*}, \, C = \ncl{M_2^*M_2}, \, D = \ncl{M_1^*M_1}.
\]
Then there is a bijection (up to unitary equivalence) between the non-degene\-ra\-te completely contractive bimodule maps of ${}_A X_B$ and of ${}_C Y_D$.
This bijection preserves non-degenerate faithful CES representations.
\end{proposition}

\begin{proof}
Let $(\pi, \phi, \si) \colon {}_A X_B \to \B(K_1, K_2)$ be a non-degenerate completely contractive bimodule map.
Let the induced $*$-representations $\tau \colon D \to \B(H_1)$ and $\rho \colon C \to \B(H_2)$, the induced completely contractive mappings $s \colon M_1 \to \B(H_1, K_1)$ and $t \colon M_2 \to \B(H_2, K_2)$ for $H_1 := M_1^* \otimes_\si K_1$ and $H_2 := M_2^* \otimes_\pi K_2$.
Then $(\si, s, \tau)$ and $(\pi, t, \rho)$ form TRO mappings of $M_1$ and $M_2$ respectively.
The C*-identity yields
\begin{align*}
\| \sum_{i=1}^d t(n_i)^* \phi(x_i) s(m_i) \|
& = \\
& \hspace{-1cm} =
\| [\pi(n_i n_j^*)]^{1/2} \begin{bmatrix} \phi(x_1) & \dots & \phi(x_d) \\ \vdots & \vdots & \vdots \\ 0 & \dots & 0 \end{bmatrix} [\si(m_i m_j^*)]^{1/2} \|
\\
& \hspace{-1cm}\leq
\|  [n_i n_j^*]^{1/2} \begin{bmatrix} x_1 & \dots & x_d \\ \vdots & \vdots & \vdots \\ 0 & \dots & 0 \end{bmatrix} [m_i m_j^*]^{1/2}  \| 
 =
\| \sum_{i=1}^d n_i^* x_i m_i \|
\end{align*}
by using that $\phi$ is completely contractive.
Similar arguments hold for all matrix norms and thus the mapping $\psi \colon Y \to \B(H_1, H_2)$ given by
\begin{align*}
\psi(\sum_{i=1}^d n_i^* x_i m_i) \sum_{j} \xi_j^* \otimes k_j
& = 
\sum_{i,j} n_i^* \otimes \phi(x_i m_i \xi_j^*) k_j \\
& =
\sum_{i=1}^d t(n_i)^* \phi(x_i) s(m_i) \sum_j \xi_j^* \otimes k_j
\end{align*}
extends to a completely contractive mapping.
Non-degeneracy of $\rho$ and $\tau$ is automatic.
Non-degeneracy of $\psi$ follows by
\begin{align*}
[\psi(Y) H_1']
& =
[t(M_2)^* \phi(X) s(M_1) M_1^* \otimes_\si K_1'] 
 =
[t(M_2)^* \phi(X) \si(B) K_1'] \\
& =
[t(M_2)^* \phi(X) K_1']
=
[t(M_2)^* K_2'] 
 =
M_2^* \otimes_\pi K_2'
=
H_2'.
\end{align*}
Hence $(\rho, \psi, \tau)$ is a non-degenerate completely contractive bimodule mapping of ${}_C Y_D$.
The arguments above also imply that $\psi$ is completely isometric if $\phi$ is completely isometric.
Consequently $(\rho, \psi, \tau)$ is a faithful CES representation when $(\pi, \phi, \si)$ is a faithful CES representation.

For the second part of the proof suppose we apply the same construction to $(\rho, \psi, \tau)$ to obtain a $(\pi', \phi' ,\si') \colon {}_A X_B \to \B(K_1',K_2')$ with
\[
K_1' = M_1 \otimes_ \tau M_1^* \otimes_\si K_1
\qand
K_2' = M_2 \otimes_\rho M_2^* \otimes_ \pi K_2.
\]
Then $\si'$ is unitarily equivalent to $\si$ by some $V \in \B(K_1, K_1')$ such that $V^*(\eta_1 \otimes \eta_2^* \otimes k_1) = \si(\eta_1 \eta_2^*) k_1$ for $\eta_1, \eta_2 \in M_1$ and $k_1 \in K_1$.
Similarly $\pi'$ is unitarily equivalent to $\pi$ by some $U \in \B(K_2,K_2')$ such that $U \pi(\xi_1 \xi_2^*) k_2 = \xi_1 \otimes \xi_2^* \otimes k_2$ for $\xi_1, \xi_2 \in M_2$ and $k_2 \in K_2$.
Then straightforward computations for $x \in M_2 M_2^* X M_1 M_1^*$ and then using continuity imply that $\phi'(x) = U \phi(x) V^*$ for all $x \in \ncl{M_2 M_2^* X M_1 M_1^*} = X$.
\end{proof}

\begin{corollary}\label{C: key}
Let $X$ and $Y$ be strongly TRO equivalent by $M_1$ and $M_2$, and let
\[
A = \ncl{M_2M_2^*}, \, B = \ncl{M_1M_1^*}, \, C = \ncl{M_2^*M_2}, \, D = \ncl{M_1^*M_1}.
\]
Proposition \ref{P: bij rep} then yields that for every non-degenerate completely contractive bimodule map $(\pi, \phi, \si) \colon {}_A X_B \to \B(K_1, K_2)$ there exists a non-degenerate completely contractive bimodule map $(\rho, \psi, \tau) \colon {}_C Y_D \to \B(H_1, H_2)$ such that $\ncl{\phi(X)}$ and $\ncl{\psi(Y)}$ are strongly TRO equivalent.
In particular $(\pi, \phi, \si)$ is a faithful CES representation if and only if so is $(\rho, \psi, \tau)$.
\end{corollary}

\subsection{Strong $\Delta$-equivalence}

We now pass to the representation-free equivalence for operator spaces.

\begin{definition}\label{D: Delta op sp}
Let $X\subseteq B(K_1, K_2)$ and $Y\subseteq B(H_1, H_2)$ be operator spaces.
We say that $X$ and $Y$ are \emph{strongly $\Delta$-equivalent} if they have completely isometric representations $\phi$ and $\psi$ such that $\phi(X)$ is strongly TRO equivalent to $\psi(Y)$.
\end{definition}

We will use the faithful CES representations of a specific C*-bimodule.
Suppose that $\I(\S(X)) \to \B(L)$ is a non-degenerate and faithful representation.
Then \cite[Paragraph 4.6.9]{BleLeM04} yields a faithful CES representation
\[
(\wt\pi, \phi, \wt\si) \colon {}_{\A_l(X)} X_{\A_r(X)} \to \B(L_1, L_2).
\]
By Proposition \ref{P: nd compression} we may assume that $\phi$ is non-degenerate by passing to the non-degenerate compression.

\begin{lemma}\label{L: fix one get one}
Let $X\subseteq B(K_1, K_2)$ and $Y\subseteq B(H_1, H_2)$ be operator spaces that act non-degenerately.
Let
\[
(\wpi, \phi, \wsi) \colon {}_{\A_l(X)} X_{\A_r(X)} \to \B(L_1, L_2)
\]
be a non-degenerate faithful CES representation.
If $X$ is strongly TRO equivalent to $Y$ then there exists a non-degenerate completely isometric mapping $\psi$ of $Y$ such that $\phi(X)$ is strongly TRO equivalent to $\psi(Y)$.
\end{lemma}

\begin{proof}
Suppose that $X$ and $Y$ are strongly TRO equivalent by $M_1$ and $M_2$.
Then we have the non-degenerate C*-bimodules ${}_A X_B$ and ${}_C Y_D$ for
\[
A = \ncl{M_2 M_2^*}, \, B = \ncl{M_1 M_1^*}, \, C= \ncl{M_2^* M_2}, \, D = \ncl{M_1^* M_1}.
\]
Notice that $A$ and $B$ act faithfully on $X$ since they act non-degenerately on $X$ and $X$ acts non-degenerately on the Hilbert spaces.
The mapping
\[
m \colon A \times X \rightarrow X : (a, x) \mapsto ax
\]
is an oplication since $X = \ncl{AX}$.
Thus there exists a $*$-homomorphism $\al \colon A \rightarrow \A_l(X)$ satisfying $\al(a) x = ax$.
For the $*$-representation $\pi := \wpi \circ \al \colon A \to \B(L_2)$ we obtain
\[
\pi(a) \phi(x) = \wpi(\al(a)) \phi(x) = \phi(\al(a) x) = \phi(ax)
\]
for all $a\in A$ and $x \in X$.
Since $A$ acts faithfully on $X$ we have that $\pi$ is faithful.
Similarly we define a faithful $*$-representation $\si \colon B \to \B(L_1)$.
Thus we obtain a faithful CES representation $(\pi, \phi, \si) \colon {}_{A} X_{B} \to \B(L_1, L_2)$.
Lemma \ref{L: nd iff} implies that $(\pi, \phi, \si)$ is also non-degenerate.
Then Corollary \ref{C: key} induces the required $\psi$ of $Y$ which completes the proof.
\end{proof}

Now we are in position to show the main result of this section.

\begin{theorem}\label{T: eq delta}
Strong $\Delta$-equivalence of operator spaces is an equivalence relation.
\end{theorem}

\begin{proof}
Trivially strong $\Delta$-equivalence is reflexive and symmetric.
For transitivity let the completely isometric mappings $\phi_0$ of $X$, $\psi_0$ and $\theta_0$ of $Y$, and $\zeta_0$ of $Z$ such that $\phi_0(X)$ is strongly TRO equivalent to $\psi_0(Y)$ and $\theta_0(Y)$ is strongly TRO equivalent to $\zeta_0(Z)$.
By applying Proposition \ref{P: nd compression} to $\psi_0$ and $\theta_0$ and then Corollary \ref{C: key} we may assume that all mappings are non-degenerate.

Fix a faithful CES representation $(\wt\rho, \psi, \wt\tau)$ of ${}_{\A_l(Y)} Y_{\A_r(Y)}$ such that $\psi$ is non-degenerate.
Then Lemma \ref{L: fix one get one} applies to $\psi \circ \psi_0^{-1}$ of $\psi_0(Y)$ to give a non-degenerate completely isometric map $\phi_1$ of $\phi_0(X)$ so that $\phi_1(\phi_0(X))$ is strongly TRO equivalent to $\psi(Y)$.
Similarly Lemma \ref{L: fix one get one} applies to $\psi \circ \theta_0^{-1}$ of $\theta_0(Y)$ to produce a non-degenerate completely isometric mapping $\zeta_1$ of $\zeta_0(Z)$ so that $\zeta_1(\zeta_0(Z))$ is strongly TRO equivalent to $\psi(Y)$.
Then transitivity of strong TRO equivalence given by Theorem \ref{T: eq rel} shows that $\phi_1(\phi_0(X))$ and $\zeta_1(\zeta_0(Z))$ are strong TRO equivalent, and the proof is complete.
\end{proof}

\begin{remark}
In \cite{Ele14} the first author introduces the notion of strong $\Delta$-equivalence for operator algebras.
Two operator algebras $X$ and $Y$ are called \emph{strongly $\Delta$-equivalent as operator algebras} if there exist completely isometric homomorphisms $\phi$ and $\psi$, and a TRO $M$ such that
\[
\phi(X) = \ncl{M \psi(Y) M^*} \qand \psi(Y) = \ncl{M^* \phi(X) M}.
\]
Consequently if two operator algebras are strongly $\Delta$-equivalent in the sense of \cite{Ele14} then they are so as operator spaces.
However the converse fails.

For a counterexample, define the operator algebras
\[
X = \{ \begin{bmatrix} 0 & \la \\ 0 & 0 \end{bmatrix} \in M_2(\bC) \mid \la \in \bC \}
\qand
Y = \{ \begin{bmatrix} 0 & 0 \\ 0 & \la \end{bmatrix} \in M_2(\bC) \mid \la \in \bC\}.
\]
They are are strongly $\Delta$-equivalent as operator spaces by the TRO's
\[
M_1 = \bC I_2 \qand M_2 = \{ \begin{bmatrix} 0 & \la \\ 0  & 0 \end{bmatrix} \in M_2(\bC) \mid \la \in \bC\}
\]
but they are not strongly $\Delta$-equivalent as operator algebras.
If they were, then their diagonals $\Delta(X):= X \cap X^* = 0$ and $\Delta(Y) := Y \cap Y^* = Y$ would also be strongly $\Delta$-equivalent as operator algebras by \cite[Theorem 2.3]{Ele14}, which is a contradiction.
\end{remark}

\section{Stable Isomorphism} \label{S: stable}

As usual we write $\K(K_1, K_2)$ for the compact operators in $\B(K_1, K_2)$ and $\K$ for $\K(\ell^2)$.
For an operator space $X \subseteq \B(K_1, K_2)$ we denote by $\C_\infty(X)$ the space of columns
\[
[x_1, x_2, \dots]^t \in \B(K_1, K_2^{(\infty)})
\]
where $x_i \in X$ and the sequence $(\sum_{i=1}^n x_i^* x_i)$ converges in norm topology.
Similarly $\R_\infty(X)$ denotes the space of the rows
\[
[x_1, x_2, \dots] \in \B(K_1^{(\infty)}, K_2)
\]
where $ x_i \in X$ and the sequence $(\sum_{i=1}^n x_i x_i^*)$ converges in norm topology.
Finally $\M_\infty(X)$ is the set of bounded operators from $K_1^{(\infty)}$ to $K_2^{(\infty)}$ which can be represented as infinite matrices with entries in $X$.
We write $\K_\infty(X) $ for the norm closure of finitely supported matrices in $\M_\infty(X)$.
Then we obtain the following identifications
\[
\K_\infty(X)  \simeq \R_\infty(\C_\infty(X)) \simeq \C_\infty(\R_\infty(X)) \simeq X \otimes \K
\]
by completely isometric isomorphic maps.
We underline here the difference between the notation $\K_\infty(X)$ and the notation $\K(X)$ of Example \ref{E: compacts}.

\begin{definition}\label{D: stably}
We call two operator spaces $X$ and $Y$ \emph{stably isomorphic} if the spaces $\K_\infty(X)$ and $\K_\infty(Y)$ are
completely isometrically  isomorphic.
\end{definition}

It is immediate that stable isomorphism of operator spaces is an equivalence relation.

\begin{remark}
Definition \ref{D: stably} coincides with the stable isomorphism of \cite{BMP00} for operator algebras with c.a.i.'s.
This follows by the fact that if two operator algebras with c.a.i.'s are completely isometric isomorphic as operator spaces then they are so as operator algebras; see for example \cite[Proposition 4.5.13]{BleLeM04}.
In particular Definition \ref{D: stably} coincides with the usual stable isomorphism for C*-algebras of \cite{BGR77}, as contractive homomorphisms between C*-algebras are automatically positive.
\end{remark}

\begin{theorem} \label{T: stable gives delta}
Stably isomorphic operator spaces are strongly $\Delta $-equivalent.
\end{theorem}

\begin{proof}
Every operator space $X \subseteq \B(K_1, K_2)$ is strongly TRO equivalent (thus strongly $\Delta$-equivalent) to $\K_\infty(X)$ for the TRO's $M_1 = \K(K_1, K_1^{(\infty)})$ and $M_2 = \K(K_2, K_2^{(\infty)})$. 
This follows by writing $M_1 = \bC \otimes \C_\infty(\bC)$ acting on $K_1 \otimes \bC$, $M_2 = \bC \otimes \C_\infty(\bC)$ acting on $K_2 \otimes \bC$, and by using that $X \simeq X \otimes \bC$.
Now if $\K_\infty(X)$ is completely isometrically isomorphic to $\K_\infty(Y)$ then they are strongly $\Delta$-equivalent and transitivity from Theorem \ref{T: eq delta} shows that $X$ is strongly $\Delta$-equivalent to $Y$.
\end{proof}

\begin{remark}
The proof of Theorem \ref{T: stable gives delta} applies also for showing that if $X$ and $Y$ are operator spaces such that $X \otimes \K(H) \simeq Y \otimes \K(H)$ for a Hilbert space $H$, then $X$ and $Y$ are strongly $\Delta $-equivalent.
\end{remark}

\begin{remark}
A converse of Theorem \ref{T: stable gives delta} is not expected in full generality.
Indeed strongly Morita equivalent C*-algebras are strongly TRO equivalent and thus strongly $\Delta$-equivalent.
However there are examples of strongly Morita equivalent C*-algebras that are not stably isomorphic \cite{BGR77}.
\end{remark}

Nevertheless we show that stable isomorphism coincides with strong $\Delta$-equivalen\-ce under the usual separability condition.
The key is the existence of $\si$-units as in \cite{BGR77, Ele14}.

\begin{theorem}\label{T: key}
Let $X$ and $Y$ be strongly TRO equivalent by $M_1$ and $M_2$.
Suppose there exist sequences $(m_i)_{i\in \bN}, (n_i)_{i\in \bN}\subseteq M_2$ such that $\sum_{i=1}^km_i m_i^*$ and $\sum_{i=1}^kn^*_in_i$ are contractions for all $k$ that satisfy
\[
\lim_k \sum_{i=1}^k m_i m_i^* x = x \foral x \in X \, \text{ and } \,
\lim_k\sum_{i=1}^k n_i^*n_i y=y  \foral y \in Y,
\]
and sequences $(e_i)_{i\in \bN}, (f_i)_{i\in \bN}\subseteq M_1$ such that $\sum_{i=1}^k e_ie_i^*$ and $\sum_{i=1}^k f_i^*f_i$ are contractions for all $k$ that satisfy
\[
\lim_k x \sum_{i=1}^k e_ie_i^* = x \foral x \in X \, \text{ and } \,
\lim_k y \sum_{i=1}^k f_i^*f_i = y \foral y \in Y.
\]
Then $X$ and $Y$ are stably isomorphic.
\end{theorem}

\begin{proof}
We will make use of a standard argument of absorption; see for example \cite{Ble96} and \cite[Corollary 8.2.6]{BleLeM04}.
For convenience set
\[
A = \ncl{M_2 M_2^*}, \, B = \ncl{M_1 M_1^*}, \, C = \ncl{M_2^* M_2}, \, D = \ncl{M_1^* M_1}
\]
which are $\si$-unital C*-algebras by assumption.
First we show that
\[
\C_\infty(Y) \simeq \C_\infty(Z) \oplus^\perp \C_\infty(Y) \qfor Z = \ncl{XM_1}.
\]
To this end let the completely contractive mappings
$\phi \colon Z \to \C_\infty(Y)$ and $\psi \colon \C_\infty(Y) \to Z$ such that
\[
\phi(z) = [m_i^* z] \qand  \psi([y_i]) = \sum_i m_i y_i.
\]
Then $\psi \circ \phi = \id_Z$ and therefore $\phi$ is a complete isometry, $\psi$ is onto $Z$ and $p : = \phi \circ \psi$ is an idempotent.
In particular $p$ is a projection in the adjointable operators on $\C_\infty(Y)$, where the latter is seen as a right Hilbert $D$-module.
Hence $p(Z)$ is orthocomplemented in $\C_\infty(Y)$.
Since $\ran p = \phi(Z)$, we get that $\C_\infty(Y) = \phi(Z) \oplus^\perp \ran(1-p)$.
Note that $\phi(g)^* \phi(g) = g^*g$,
hence $\C_\infty(Y) \simeq Z \oplus^\perp \ran(1-p)$.
Then the absorption argument yields
\begin{align*}
\C_\infty(Y)
& \simeq \C_\infty(\C_\infty(Y)) 
\simeq \C_\infty(Z) \oplus^\perp Z \oplus^\perp \ran(1-p)
\simeq \C_\infty(Z) \oplus^\perp \C_\infty(Y).
\end{align*}
Similarly one can show that $\C_\infty(Y) \simeq \C_\infty(Z)$ for the maps
\[
Y \to \C_\infty(Z) : y \mapsto [n_i y]
\qand
\C_\infty(Z) \to Y : [z_i] \mapsto \sum_i n_i^* z_i.
\]
and by interchanging the roles of $Y$ and $Z$ in the arguments above.

To end the proof observe that $Z$ is strongly TRO equivalent to $X$ by $\bC$ and $M_1$.
By applying with $Z$ in the place of $Y$ above we also get that $\C_\infty(X) \simeq \C_\infty(Z)$.
We derive that $\C_\infty(X) \simeq \C_\infty(Y)$ and therefore
\[
\K_\infty(X) \simeq \R_\infty(\C_\infty(X)) \simeq \R_\infty(\C_\infty(Y)) \simeq \K_\infty(Y),
\]
which completes the proof.
\end{proof}

\begin{corollary}\label{C: si-unital}
Let $X$ and $Y$ be strongly TRO equivalent by $M_1$ and $M_2$.
If the C*-algebras
\[
A = \ncl{M_2M_2^*}, \, B = \ncl{M_1M_1^*}, \, C = \ncl{M_2^*M_2}, \, D = \ncl{M_1^*M_1}
\]
admit countable approximate identities then $X$ and $Y$ are stably isomorphic.
\end{corollary}

\begin{proof}
By definition the C*-algebra $\ncl{M_i^* M_i}$ is strongly Morita equivalent to $\ncl{M_iM_i^*}$ for $i=1,2$.
Therefore by a standard argument for TRO's  \cite[Lemma 2.3]{Bro77} there are sequences $(m_i)_{i\in \bN}, (n_i)_{i\in \bN}$ in $M_2$, and $(e_i)_{i\in \bN}, (f_i)_{i\in \bN}\subseteq M_1$ such that
\[
\| \sum_{i=1}^k m_i m_i^* \| \leq 1, \,
\| \sum_{i=1}^k e_i e_i^* \| \leq 1, \,
\| \sum_{i=1}^k n^*_i n_i \| \leq 1, \,
\| \sum_{i=1}^k f_i^* f_i \| \leq 1,
\]
for all $k \in \bN$, and they form countable c.a.i.'s for $A$, $B$, $C$ and $D$ respectively.
The proof then follows by applying Theorem \ref{T: key}.
\end{proof}

\begin{corollary}\label{C: sep}
Let $X$ and $Y$ be separable operator spaces.
Then they are strongly $\Delta$-equivalent if and only if they are stably isomorphic.
\end{corollary}
\begin{proof}
By Theorem \ref{T: stable gives delta} it suffices to prove the forward implication.
Notice that if $\phi$ is a completely isometric map of $X$ then $\K_\infty(X) \simeq \K_\infty(\phi(X))$.
Hence without loss of generality, we may assume that $X$ and $Y$ are strongly TRO equivalent by $M_1$ and $M_2$.
Define the separable operator space
\[
\T(X) = \ncl{x_1x_2^*x_3x_4^*\cdots x_{2n}^* x_{2n+1} \mid n\geq 0, x_1,\dots, x_{2n+1} \in X}.
\]
Therefore the $C^*$-algebras
\[
\A := \ncl{\T(X)\T(X)^*} \qand \B := \ncl{\T(X)^*\T(X)}
\]
are separable as well and thus they possess countable c.a.i.'s.
Similar observations hold for the C*-algebras
\[
\C := \ncl{\T(Y) \T(Y)^*} \qand \D := \ncl{\T(Y)^*\T(Y)}.
\]
Then by \cite[Lemma 2.3]{Bro77} (see also \cite[Lemma 3.4]{Ele14}), there exist sequences $(m_i)_{i \in \bN}, (n_i)_{i \in \bN} \subseteq M_2$ and $(e_i)_{i \in \bN}, (f_i)_{i \in \bN} \subseteq M_1$ that satisfy the assumptions of Theorem \ref{T: key}.
Hence $X$ and $Y$ are stably isomorphic.
\end{proof}

We aim to show that stable isomorphism and strong $\Delta$-equivalence coincide for the class of unital operator spaces.
We require the next lemma.

\begin{lemma}\label{L: unital}
Let $X$ be a unital operator space and let $\phi \colon X \to \B(K_1, K_2)$ be a (not necessarily unital) non-degenerate completely isometric map.
If $A \subseteq \B(K_2)$ and $B \subseteq \B(K_1)$ are C*-algebras that act non-degenerately, and
\[
\phi(X) = \ncl{A \phi(X)} = \ncl{\phi(X) B},
\]
then $A$ and $B$ are unital.
\end{lemma}

\begin{proof}
We will show that $A$ is unital.
The case for $B$ follows by similar arguments.
For the first step, notice that the completely contractive mapping
\[
m \colon A \times X \to X : (a,x) \mapsto \phi^{-1}(a \phi(x))
\]
is an oplication since $\phi(X) = \ncl{A \phi(X)}$.
Hence there is a $*$-homomorphism $\al \colon A \to \A_l(X)$ such that
\[
\al(a) (x) = \phi^{-1}(a\phi(x)) \foral a \in A, x \in X.
\]
In particular $\al$ is faithful since $K_2 = [\phi(X) K_1]$.

For the second step, let $\I(X) = \I(X + X^*)$ be the injective envelope of the unital space $X$.
Recall from \cite[Proposition 4.4.13 and Proposition 4.5.8]{BleLeM04} that $\A_l(X)$ is $*$-isomorphic to the unital C*-subalgebra
\[
\{s \in \I(X) \mid s X \subseteq X \text{ and } s^* X \subseteq X\}
\]
of $\I(X)$ by the $*$-homomorphism $\wt\pi \colon \A_l(X) \to \I(X)$ that satisfies $\wt\pi(u) x = u(x)$ for all $u \in \A_l(X), x \in X$.
Set $\pi = \wt\pi\circ \al$.
Then we have that
\[
\pi(a) x = \wt\pi \circ \al(a) x = \phi^{-1}(a \phi(x))
\]
for all $a \in A$ and $x \in X$.
Since $\phi(X) = \ncl{A \phi(X)}$ we obtain that
\[
X = \ncl{\phi^{-1}(A \phi(X))} = \ncl{\pi(A) X}.
\]
Thus we get that $I_X \in \ncl{\pi(A) X}$.
If $(a_i)$ is a c.a.i. of $A$ we then obtain
\[
\lim_i \pi(a_i) \pi(a) x = \pi(a) x \foral a \in A, x \in X,
\]
thus $\lim_i \pi(a_i) I_X = I_X$.
Hence we get that $I_X \in \pi(A)$.
However the embedding $X \hookrightarrow \I(X)$ is unital and so $I_X$ is also the unit of the C*-algebra $\I(X)$.
Thus it is also the unit of its unital C*-subalgebra $\wt\pi(\A_l(X))$.
Since $\pi(A) \subseteq \wt\pi(\A_l(X))$ we get that $\pi(A)$ is unital; thus so is $A$ as $\pi$ is faithful.
\end{proof}

\begin{corollary}\label{C: unital}
Two unital operator spaces are strongly $\Delta$-equivalent if and only if they are stably isomorphic.
\end{corollary}

\begin{proof}
Because of Theorem \ref{T: stable gives delta}, it suffices to show the forward implication.
To this end let the completely isometric mappings $\phi$ of $X$ and $\psi$ of $Y$ such that $\phi(X)$ is strongly TRO equivalent to $\psi(Y)$ by $M_1$ and $M_2$.
Note that $\phi$ and $\psi$ may not be unital.
However by applying Proposition \ref{P: nd compression} to $\phi$ and then Corollary \ref{C: key} we may assume without loss of generality that they are non-degenerate.
Then Lemma \ref{L: unital} implies that the C*-algebras
\[
A = \ncl{M_2M_2^*}, B = \ncl{M_1M_1^*}, C = \ncl{M_2^*M_2}, D = \ncl{M_1^*M_1}
\]
are unital, and Corollary \ref{C: si-unital} finishes the proof.
\end{proof}

\begin{example}\label{E: trivial sep}
Recall the spaces $X_0, X, Y, M_1, M_2$ from Example \ref{E: trivial}.
If $M_1$ and $M_2$ are separable then so are the C*-algebras $\ncl{M_i^*M_i}$ and $\ncl{M_iM_i^*}$ for $i=1, 2$.
Therefore $X$ and $Y$ are stably isomorphic.
\end{example}

\begin{example}\label{E: compacts sep}
Let $X$ and $Y$ be dual operator spaces that act on separable Hilbert spaces.
By Example \ref{E: compacts} the spaces $\K(X)$ and $\K(Y)$ are strongly TRO equivalent when $X$ and $Y$ are weakly TRO equivalent.
Since $\K(X)$ and $\K(Y)$ are separable then Corollary \ref{C: sep} implies that they are stably isomorphic.
Similar arguments show that $\F(X)$ and $\F(Y)$, and $\R(X)$ and $\R(Y)$ are stably isomorphic as well.
\end{example}

\begin{example}\label{E: lattices sep}
Let $\L_1, \L_2$ be reflexive lattices acting on separable Hilbert spaces.
Let $A = \Alg(\L_1)$ and $B = \Alg(\L_2)$ be the corresponding algebras.
If $\theta \colon \L_1'' \to \L_2''$ is a $*$-isomorphism such that $\theta (\L_1)=\L_2$ then the algebras $A$ and $B$ are weakly TRO equivalent \cite{Ele12}.
Therefore by Example \ref{E: lattices}, Example \ref{E: compacts sep} and Corollary \ref{C: sep} the algebras $\K(A)$ and $\K(B)$ are stably isomorphic.
Likewise the algebras $\F(A)$ and $\F(B)$, and the algebras $\R(A)$ and $\R(B)$  are stably isomorphic.
\end{example}

\section{Applications}\label{S: app}

\subsection{Strong $\Delta$-equivalence and operator algebras}

Strong $\Delta$-equiva\-len\-ce coincides with strong Morita equivalence in the sense of Rieffel \cite{Rie74} for the class of C*-algebras.
In fact we have the following result.

\begin{theorem}\label{T: delta gives BMP}
Let $X$ and $Y$ be operator algebras with c.a.i.'s.
If they are strongly $\Delta$-equivalent as operator spaces then they are strongly Morita equivalent in the sense of Blecher, Muhly and Paulsen \cite{BMP00}.
\end{theorem}

Combining Theorem \ref{T: delta gives BMP} with \cite[Theorem 6.2]{BMP00} gives the next corollary.

\begin{corollary}\label{C: same in C*}
Let $X$ and $Y$ be C*-algebras.
Then they are strongly $\Delta$-equivalent as operator spaces if and only if they are strongly Morita equivalent as C*-algebras.
\end{corollary}

\begin{remark} \label{R: delta gives BMP strict}
The implication of  Theorem \ref{T: delta gives BMP} is strict.
In \cite[Example 8.2]{BMP00}, Blecher, Muhly and Paulsen construct unital operator algebras which are strongly Morita equivalent but not stably isomorphic.
However strong $\Delta$-equivalence coincides with stable isomorphism for unital operator algebras as shown in Corollary \ref{C: unital}.
\end{remark}

The proof of Theorem \ref{T: delta gives BMP} will follow from a series of steps and reductions.
For the first reduction, let $X$ and $Y$ be strongly TRO equivalent operator spaces and let $\phi\colon X \to \B(K)$ and $\psi \colon Y \to \B(H)$ be completely isometric maps such that $\phi(X)$ and $\psi(Y)$ are operator algebras with c.a.i.'s.
Then it suffices to show that $\phi(X)$ and $\psi(Y)$ are strongly Morita equivalent in the sense of \cite{BMP00}, i.e. we have to find a Morita context for $\phi(X)$ and $\psi(Y)$.
To this end let $M_1$ and $M_2$ be TRO's such that
\[
X = \ncl{M_2 Y M_1^*} \qand Y = \ncl{M_2^* X M_1}.
\]
Moreover let
\[
A = \ncl{M_2M_2^*}, \, B = \ncl{M_1M_1^*}, \, C = \ncl{M_2^*M_2}, \, D = \ncl{M_1^*M_1}.
\]
We aim to show that the spaces
\begin{align*}
E = \ncl{M_2 Y} = \ncl{X M_1} \qand F = \ncl{M_2^* X} = \ncl{Y M_1^*}
\end{align*}
provide the appropriate Morita context.

\begin{lemma} \label{L: Mc eq}
With the aforementioned notation, we have that
\[
X \simeq E \otimes ^h_D M_1^* \simeq M_2 \otimes_C^h F \qand Y \simeq F \otimes ^h_B M_1 \simeq M_2^* \otimes_A^h E.
\]
Consequently we obtain
\[
E\simeq X\otimes^h_B M_1 \simeq  M_2\otimes ^h_C Y \qand F\simeq Y \otimes^h_D M_1^*\simeq M_2^*\otimes ^h_A X.
\]
\end{lemma}

\begin{proof}
By using that $E = \ncl{M_2Y}$ and the completely contractive and $D$-balanced bilinear map
$(m_2 y, m_1^*) \mapsto m_2 y m_1^*$ we obtain the completely contractive map
\[
\theta \colon E \otimes ^h_D M_1^* \mapsto X: e \otimes m_1^* \to e m_1^*.
\]
Let $v(\la)^* = [(n_1^\la)^*, \dots, (n_{k_\la}^\la)^*]$ be a net of row contractions such that  $n_i^\la \in M_1$ and $\lim_\la m v(\la)^*v(\la) = m$ for all $m\in M_1$.
Fix $e_1, \dots, e_n\in E$ and $m_1, \dots, m_n \in M_1$.
For $\eps >0$ there exists $\la$ such that
\begin{align*}
\| \sum_{i=1}^n e_i \otimes m_i \| -\eps
& \leq
\| \sum_{i=1}^n e_i \otimes(m_i v(\la)^* v(\la)) \| \\
& \leq
\| \sum_{i=1}^n e_i m_i v(\la)^* \|
\leq
\| \sum_{i=1}^n e_i m_i \|.
\end{align*}
Thus $\theta $ is isometric. Similarly we can prove that it is completely isometric.
The proofs of the other assertions in the first part follow in a similar way.
For the second part we have that
\[
X \otimes_B^h M_1 \simeq E \otimes_D^h (M_1^* \otimes_B^h M_1) \simeq E \otimes_D^h D \simeq E.
\]
The other assertions follow in a similar way and the proof is complete.
\end{proof}

\begin{lemma}\label{L: Mc *-hom}
Let ${}_A X_B$ be a C*-bimodule with $X = \ncl{AX} = \ncl{XB}$.
Let $\phi \colon X \to \B(K)$ be a completely contractive map such that $\phi(X)$ is an operator algebra with a c.a.i. .
Then there exist $*$-homomorphisms $\pi \colon A \to \B(K)$ and $\si \colon B \to \B(K)$ such that
\[
\pi(a) \phi(x) = \phi(ax) \qand \phi(x) \si(b) = \phi(xb)
\]
for all $a\in A, b \in B, x \in X$.
\end{lemma}

\begin{proof}
It suffices to give the proof for $\pi$.
First of all observe that the mapping $m \colon A \times \phi(X) \to \phi (X)$ with $m(a,\phi (x)) = \phi (ax)$ is an oplication.
Hence there exists a $*$-homomor\-phism $\al \colon A \to \A_l(\phi (X))$ such that $\al(a) (\phi (x)) =\phi(ax)$ for all $a\in A, x\in X$.
Let the completely isometric isomorphism
\[
\wt\pi \colon \M_l(\phi(X)) \to \{s \in B(K) \mid s \phi(X) \subseteq \phi(X)\}
\]
such that $\wt\pi(L) \phi(x) = L(\phi(x))$ for all $L \in \M_l(\phi(X)), x \in X$.
Then $\pi := \wt\pi \circ \al \colon A \to \B(K)$ is a completely contractive homomorphism and satisfies
\[
\pi(a) \phi(x) = \wt\pi(\al(a)) \phi(x) = \al(a)(\phi(x)) = \phi(ax),
\]
for all $a\in A$ and $x \in X$.
Completely contractive homomorphisms on C*-algebras are automatically $*$-homomorphisms, and the proof is complete.
\end{proof}

\begin{lemma}\label{L: Mc bm}
Let $E$ and $F$ be as above.
Then the space $E$ admits an operator $\phi(X)$-$\psi(Y)$-bimodule structure, and the space $F$ admits an operator $\psi(Y)$-$\phi(X)$-bimodule structure.
\end{lemma}

\begin{proof}
Recall that $F \simeq M_2^* \otimes_A^h X$, and define the operation
\[
(m^* \otimes x) \odot \phi(w) = m^* \otimes \phi^{-1}(\phi(x) \phi(w))
\]
for $m \in M_2$ and $x, w \in X$.
We will show that it defines a right module structure.
The other assertions follow in a similar way.

Let the $*$-homomorphism $\pi \colon A \rightarrow B(K)$ induced by Lemma \ref{L: Mc *-hom}.
Let $a\in A$ and for $x, w \in X$ let $z \in X$ such that $\phi(z)=\phi (x)\phi(w)$.
Then we get
\[
\phi (ax) \phi(w)=\pi(a) \phi(x) \phi(w) = \pi(a) \phi(z) = \phi(a z).
\]
Consequently we obtain $a z = \phi^{-1} (\phi(a x) \phi (w))$ and therefore
\begin{equation}\label{exxx}
a \phi^{-1}(\phi(x) \phi (w)) = \phi^{-1} (\phi (ax) \phi(w)).
\end{equation}
Similar arguments imply the matrix version equation
\[
[a_{ij}]\cdot \phi_{n}^{-1} \left( \phi_n([x_{ij}]) \phi_n([w_{ij}]) \right) = \phi^{-1}_n\left( \phi_n([a_{ij}] [x_{ij}]) \phi_n([w_{ij}]) \right)
\]
for all $n \in \bN$.
In particular by regarding every rectangular matrix as a submatrix of an appropriate square matrix with zeroes we may extend these formulas to cover appropriate rectangular cases.

Let $u(t)^* = [(m_1^t)^*, \dots, (m_{l_t}^t)^*]$ be a net of row contractions such that  $m_i^t \in M_2$ and $\lim_\la u(t)^* u(t) m^* = m^*$ for all $m\in M_2$.
Fix $m\in M_2$ and $x, w \in X$.
For $\eps > 0$ there exists $t$ such that
\[
\|m^*\otimes \phi ^{-1}(\phi (x)\phi(w)) \| -\eps
\leq
\|(u(t)^* u(t) m^*) \otimes \phi ^{-1}(\phi (x)\phi (w)) \|.
\]
We denote by $\phi_{n,k}$ the entry-wise application of $\phi$ on a $n \times k$ matrix.
Using then the rectangular version of equation (\ref{exxx}) we have
\begin{align*}
\|m^* \otimes \phi^{-1}(\phi (x)\phi(w)) \| - \eps
& \leq
\| u(t)^* \otimes u(t) m^* \phi ^{-1}(\phi (x)\phi (w)) \| \\
& \leq
\| \phi^{-1}_{l_t,1} \left( \phi_{l_t,1}(u(t) m^*x) \phi(w) \right) \| 
\leq
\|m^*x\| \|w\|.
\end{align*}
By Lemma \ref{L: Mc eq} we have that $\|m^*x\|= \|m^*\otimes x\|$ and thus
\[
\|m^* \otimes \phi ^{-1}(\phi(x) \phi(w)) \|
\leq
\|m^*\otimes x\| \|w\|.
\]
Hence the map 
\[
F \times \phi (X) \to F : (m^*\otimes x, \phi (w)) \mapsto m^*\otimes \phi ^{-1}(\phi (x)\phi (w))
\]
is contractive.
Showing that this map is completely contractive follows by similar arguments and the proof is complete.
\end{proof}

\begin{lemma} \label{L: Mc eq 2}
With the aforementioned notation we have that
\[
F\simeq F\otimes _{\phi (X)}^h X \qand E\simeq E\otimes ^h_{\psi (Y)} Y.
\]
\end{lemma}

\begin{proof}
Lemma \ref{L: Mc bm} and Cohen's Factorization Theorem imply that $F$ is a non-degenerate right module over $\phi(X)$.
As a consequence \cite[Lemma 2.5]{BMP00} applies to obtain that $F\otimes _{\phi (X)}^h \phi(X) \simeq F$.
Analogous arguments apply to show that $E \simeq E\otimes ^h_{\psi (Y)} Y$.
\end{proof}

\noindent \textbf{Proof of Theorem \ref{T: delta gives BMP}.}
Using Lemmas \ref{L: Mc eq} and \ref{L: Mc eq 2} we have that
\begin{align*}
F\otimes ^h_{\phi(X)} E
 \simeq
(F\otimes ^h_{\phi(X)} X) \otimes ^h_B M_1
 \simeq
F\otimes ^h_B  M_1
\simeq
Y.
\end{align*}
Similarly we get that $E \otimes^h_{\psi(Y)} F \simeq X$.
Hence $\phi(X)$ and $\psi(Y)$ are strongly Morita equivalent in the sense of \cite{BMP00}.
\hfill{$\qedsymbol$}

\subsection{Second duals}

Recall from \cite[Definition 2.2(ii)]{EPT10} that two dual operator spaces $X$ and $Y$ are called \emph{$\Delta$-equivalent as dual operator spaces} if there are completely isometric normal representations $\phi$ and $\psi$, and TRO's $M_1$ and $M_2$ such that
\[
\phi(X) = \wcl{M_2 \psi(Y) M_1^*} \qand \psi(Y) = \wcl{M_2^* \phi(X) M_1}.
\]
Furthermore $X$ and $Y$ are called \emph{stably isomorphic as dual operator spaces} if there is a cardinal $J$ and a w*-continuous completely isometric map from $M_J(X)$ onto $M_J(Y)$ \cite[Definition 2.2(iii)]{EPT10}.
Recall that if $X \subseteq \B(K)$ then $M_J(X)$ denotes the subspace of $M_J(\B(K)) \simeq \B(K \otimes \ell^2(J))$ with entries from $X$.
In \cite[Theorem 2.5]{EPT10} it is shown that $\Delta$-equivalence coincides with stable isomorphism for dual operator spaces.

We aim to show the connection between the second duals of strongly $\Delta$-equivalent operator spaces.
We begin with a lemma.

\begin{lemma} \label{L: weak}
Let $M$ be a w*-closed TRO and suppose that $Y$ is a left dual operator module over $C = \wcl{M^*M}$ such that $Y = \wcl{C Y}$.
Then $Y$ is $\Delta$-equivalent to the normal Haagerup tensor product $E = M \otimes ^{\sigma  h}_C Y$.
\end{lemma}

\begin{proof}
Fix a w*-continuous completely isometric map $\psi \colon Y \to B(H_1, H_2)$, and a w*-continuous injective $*$-homomorphism $\rho \colon C \to B(H_2)$ such that
\[
\psi (cy) = \rho(c) \psi (y) \foral c\in C, y\in Y.
\]
Fix the faithful TRO morhism $s \colon M \to B(H_2, M \otimes_\rho H_2)$ and let the bilinear w*-continuous completely contractive $C$-balanced map
\[
M \times Y \to B(H_1, M \otimes_\rho H_2): (m,y) \mapsto s(m) \psi (y).
\]
This map induces a w*-continuous completely contractive map
\[
\theta \colon E \to B(H_1, M \otimes_\rho H_2) : m \otimes y \mapsto s(m) \psi (y).
\]
We aim to show that $\theta$ is completely isometric.
Let $w(t) = [m^t_1, \dots, m^t_{l_t}]$ be a net of row contractions such that $m^t_i \in M$ satisfying $\lim_t w(t) w(t)^*m = m$ for all $m \in M$.
Then for $z = \sum_{i=1}^n m_i \otimes y_i \in E$ and $\eps > 0$ there exists a $t$ such that
\begin{align*}
\nor{z} - \eps
& \leq
\| \sum_{i=1}^n w(t) w(t)^*m_i \otimes y_i \| 
\leq
\| \sum_{i=1}^n w(t)^* m_i y_i  \|.
\end{align*}
Since $s$ is completely contractive we have that
\begin{align*}
\| \psi_{l_t, 1} (\sum_{i=1}^n w(t)^* m_i y_i ) \|
& =
\| s_{l_t, 1} (w(t))^* \sum_{i=1}^n s (m_i) \psi (y_i ) \| 
 \leq
\| \sum_{i=1}^n s (m_i) \psi (y_i )\|
\end{align*}
where $\psi_{l_t, 1}$ is the application of $\psi$ entry-wise on the $l_t \times 1$ matrix, and likewise for $\rho_{l_t, 1}$ and $s_{l_t, 1}$.
Since $\eps > 0$ is arbitrary we have $\nor{z} = \nor{\theta (z)}$.
The proof that $\theta$ is completely isometric follows in a similar way.

The Krein-Smulian Theorem yields that $\theta$ has a w*-closed range, hence $\theta (E) = \overline{\theta (E)}^{w^*} = \wcl{s(M) \psi (Y)}$;
see for example \cite[Theorem A.2.5]{BleLeM04}.
In addition we have that
\begin{align*}
\wcl{s (M)^* \theta (E)}
& =
\wcl{s (M)^* s (M) \psi (Y)} 
 =
\wcl{\rho(M^*M) \psi (Y)} \\
& =
\wcl{\rho(C) \psi (Y)} 
 =
\psi (\wcl{C Y})
=
\psi(Y).
\end{align*}
Since $s(M)$ is a TRO then $E$ and $Y$ are $\Delta$-equivalent.
\end{proof}

\begin{theorem}\label{T: second dual}
Strongly $\Delta$-equivalent operator spaces admit $\Delta$-equivalent second duals in the sense of \cite{EPT10}.
\end{theorem}

\begin{proof}
Let $X$ and $Y$ be two strongly $\Delta$-equivalent operator spaces.
We write $X^{\#\#}$ and $Y^{\#\#}$ for their second duals.
Without loss of generality assume that $X = \ncl{M_2 Y M_1^*}$ and $Y = \ncl{M_2^* X M_1}$ for the TRO's $M_1$ and $M_2$, and let
\[
A = \ncl{M_2M_2^*}, \, B = \ncl{M_1M_1^*}, \, C = \ncl{M_2^*M_2}, \, D = \ncl{M_1^*M_1}.
\]
Define $E = \ncl{M_2 Y}$.
By Lemma \ref{L: Mc eq} we have that $E \simeq M_2 \otimes^h_C Y$.
By \cite[Section 3, Example 6]{BleKas08} we get that $E^{\#\#}$ is w*-isomorhic to the normal Haagerup tensor product $M_2^{\#\#} \otimes^{\si h}_{C^{\#\#}} Y^{\#\#}$.
Then
\[
C^{\#\#} = \wcl{(M_2^{\#\#})^* M_2^{\#\#}},
\]
since $M_2^{\#\#}$ is a w*-closed TRO.
Lemma \ref{L: weak} implies that $Y^{\#\#}$ and $E^{\#\#}$ are $\Delta$-equivalent as dual operator spaces.

Similarly, $(E^*)^{\#\#}$ is $\Delta$-equivalent to $(X^*)^{\#\#}$.
It follows that $E^* \simeq M_1^* \otimes_B^h X^*$ as in the proof of  Lemma \ref{L: Mc eq}.
Thus we obtain
\[
(E^*)^{\#\#} \simeq (M_1^*)^{\#\#} \otimes_{B^{\#\#}}^{\si h} (X^*)^{\#\#}
\]
as above and Lemma \ref{L: weak} applies again.
Consequently $E^{\#\#}$ is $\Delta$-equivalent to $X^{\#\#}$ and transitivity completes the proof.
\end{proof}

\subsection{TRO envelopes}

For the next results recall the notion of the TRO envelope of an operator space from Section \ref{S: Pre}.

\begin{theorem}\label{T: tro env}
Strongly $\Delta$-equivalent operator spaces admit strongly $\Delta$-equi\-va\-lent TRO envelopes.
\end{theorem}

\begin{proof}
Suppose that $X$ and $Y$ are strongly $\Delta$-equivalent.
We will show that there are completely isometric mappings of their TRO envelopes with strongly TRO equivalent images.
By Corollary \ref{C: key}, without loss of generality we may assume that $X \subseteq \I(\S(X))$ and that $Y \subseteq \B(H_1, H_2)$ is strongly TRO equivalent to $X$ by some $M_1$ and $M_2$.
Set
\[
A = \ncl{M_2M_2^*}, \, B = \ncl{M_1M_1^*}, \,
C = \ncl{M_2^*M_2}, \, D = \ncl{M_1^*M_1}.
\]
We will also assume that $Y$ acts non-degenerately.
Since
\[
\ncl{A \tenv(X)} = \tenv(X) = \ncl{\tenv(X) B},
\]
we get that $\tenv(X)$ is strongly TRO equivalent to the space
\[
\T(Y) := \ncl{M_2^* \tenv(X) M_1}.
\]
Notice that $\T(Y)$ is a TRO extension of $Y$.
It suffices to show that $\T(Y)$ is indeed $\tenv(Y)$.

To this end let $\Phi \colon \I(\S(Y)) \to \B(H_2 \oplus H_1)$ be the completely isometric map that extends $\id \colon Y \to \B(H_1, H_2)$.
Let $Z$ be the image of $\tenv(Y)$ inside $\B(H_2 \oplus H_1)$ under $\Phi$.
Since $Z$ is a $C$-$D$-bimodule and $Y$ acts non-degenerately we get that $Z \subseteq \B(H_1, H_2)$.
Since $Z$ contains $Y$ it suffices to find a completely contractive TRO morphism $\psi \colon Z \to \T(Y)$ that fixes $Y$ pointwise.
Then the universal property of $\tenv(Y)$ completes the proof.

For convenience let $\odot$ be the multiplication rule on $\I(\S(Y))$ that turns $Z$ into a TRO extension of $Y$.
We can then define a multiplication rule, which we will denote by the same symbol $\odot$, on the space
\[
\T(X) := \ncl{M_2 Z M_1^*} = \ncl{M_2 \Phi(\tenv(Y)) M_1^*}.
\]
This multiplication rule is given by
\[
(m_1 z_1 n_1^*) \odot (n_2 z_2^* m_2^*) \odot (m_3 z_3 n_3^*)
:= m_1 (z_1 n_1^*n_2 \odot z_2^* \odot m_2^* m_3 z_3) n_3^*
\]
for $z_i \in Z$, $m_i \in M_2$, $n_i \in M_1$, and $i=1, 2, 3$.
It turns $\T(X)$ into a TRO extension of $\ncl{M_2 Y M_1^*} = X$.
Associativity follows by the fact that $(Z, \odot)$ is a TRO and also a $C$-$D$-bimodule.
Therefore there exists a completely contractive TRO morphism $\theta \colon \T(X) \to \tenv(X)$ that fixes $X$ pointwise.
Let the c.a.i.'s of $C$ and $D$ be given by $c_t = u(t)^* u(t)$ and $d_\la = v(\la)^* v(\la)$ for the row contractions
\[
u(t)^* = [ (n_1^t)^*, \dots, (n_{l_t}^t)^*] \qand v(\la)^* = [ (f_1^\la)^*, \dots, (f_{k_\la}^\la)^*].
\]
Then the nets $(c_t)$ and $(d_\la)$ form also a left unit and a right unit for $Y$.
For $t \in I$ and $\la \in \La$ let us define the completely contractive linear mapping
\[
Z \to M_{l_t, k_\la}(\T(X)): z \mapsto u(t) z v(\la)^*.
\]
Let $\theta_{t, \la} \colon M_{l_t, k_\la}(\T(X)) \to M_{l_t, k_\la}(\tenv(X))$ be the application of $\theta$ entry-wise and define the linear mapping $\psi_{t, \la}$ on $Z$ by
\[
\psi_{t, \la}(z) = u(t)^* \left( \theta_{t, \la} (u(t) z v(\la)^* \right) v(\la) = u(t)^* [\theta(m_i^t z (n_j^\la)^*)] v(\la)
\]
which takes values inside $\T(Y) = \ncl{M_2^* \tenv(X) M_1}$.
Since $\theta$ is completely contractive we get that $\psi_{t, \la}$ is a completely contractive linear mapping.
By passing to a subnet, suppose that $(\psi_{t,\la})$ converges to a completely contractive mapping $\psi$ of $Z$ in the BW-topology.

However the restriction of $\theta$ to $X = \ncl{M_2 Y M_1^*}$ is the identity mapping.
Hence we get in particular that
\begin{align*}
\psi(y)
& =
\lim_{(t, \la)} u(t) \left( u(t)^* y v(\la)^* \right) v(\la)
=
\lim_{(t, \la)} c_t y d_\la
=
y,
\end{align*}
since $(c_t)$ and $(d_\la)$ are subnets of the left and right approximate units of $Y$.
It remains to show that $\psi$ is a TRO morphism.
Then it will follow also that $\psi$ takes values inside the TRO extension $\T(Y)$ of $Y$.
Since $\theta$ is a TRO morphism and a $C$-$D$-bimodule map then so is $\theta_{t, \la}$.
Hence we obtain
\begin{align*}
\psi_{t, \la}(z_1) \psi_{t,\la}(z_2)^* \psi_{t,\la}(z_3) 
& = \\
& \hspace{-3cm} =
u(t)^* \, \theta_{t, \la} \left(u(t) \, z_1 \, (v(\la)^* v(\la))^2 \,  z_2^* \, (u(t)^* u(t))^2 \ z_3 \, v(\la)^* \right) \, v(\la).
\end{align*}
On the other hand we have that
\begin{align*}
\psi_{t,\la}(z_1 z_2^* z_3)
& =
u(t)^* \theta_{t, \la} (u(t) z_1 z_2^* z_3 v(\la)^*) v(\la).
\end{align*}
Since the $u(t)$ and $v(\la)$ are contractions, by an $\eps/2$-argument it suffices to show that
\[
\lim_{(t,\la)} (v(\la)^* v(\la))^2 z_2^* (u(t)^* u(t))^2 = z_2^*.
\]
However this is immediate as $u(t)^* u(t) = c_t$ and $v(\la)^* v(\la) = d_\la$ give respectively a left and a right c.a.i. of $Z$.
\end{proof}

Recall that the TRO envelope of a unital operator space is a C*-algebra.
The following result follows as a consequence of Theorem \ref{T: tro env}.
Nevertheless we provide an independent proof.

\begin{theorem}\label{T: cenv}
Let $X$ and $Y$ be unital operator spaces.
If $X$ and $Y$ are strongly $\Delta$-equivalent then $\cenv(X)$ and $\cenv(Y)$ are stably isomorphic as C*-algebras.
\end{theorem}

\begin{proof}
Corollary \ref{C: unital} implies that $\K_\infty(X)$ is completely isometrically isomorphic to $\K_\infty(Y)$.
Thus their TRO envelopes must be completely isometrically isomorphic as well.
Then \cite[Corollary 8.3.12]{BleLeM04} (or \cite[Appendix 1]{Ble01}) implies that $\tenv(\K_\infty(X)) \simeq \K_\infty(\tenv(X))$ and $\tenv(\K_\infty(Y)) \simeq \K_\infty(\tenv(Y))$ as TRO's.
However $X$ and $Y$ are unital and so the TRO's $\tenv(X)$ and $\tenv(Y)$ are actually C*-algebras, denoted by $\cenv(X)$ and $\cenv(Y)$.
Therefore the TRO isomorphism is a $*$-isomorphism, and so the C*-envelopes are stably isomorphic.
\end{proof}

\begin{acknow}
The authors acknowledge support from the London Ma\-thematical Society (Scheme 4, Grant Ref: 41449).
\end{acknow}


\end{document}